\documentclass[a4paper, 12pt, draft]{article}
\usepackage{amsfonts}
\usepackage{amsmath}
\usepackage{amssymb}
\usepackage{amsthm}
\usepackage{enumitem}
\usepackage{graphicx}
\usepackage{geometry}
\usepackage{multicol}

\numberwithin{equation}{section}
\newtheorem{theorem}{Theorem}[section]
\newtheorem{lemma}[theorem]{Lemma}
\newtheorem{corollary}[theorem]{Corollary}
\newtheorem{definition}[theorem]{Definition}
\setlist{nosep}
\begin{document}

\title{\bf Characterization of expansive polynomials by special determinants}

\author{M.\ J.\ Uray (Uray M.\ J\'anos) \\ uray.janos@inf.elte.hu \\ ELTE -- E\"otv\"os Lor\'and University Budapest \\ Faculty of Informatics \\ Department of Computer Algebra}
\date{}

\maketitle

\begin{abstract}
A polynomial is expansive if all of its roots lie outside the unit circle. We define some special determinants involving the coefficients of a real polynomial and formulate necessary and sufficient conditions for expansivity using these determinants. We show how these conditions can be turned into an algorithm, which, for integer polynomials, avoids exponential coefficient growth. We also examine the question how close the roots of an expansive polynomial can be to the unit circle if the coefficients are integers. We give several lower bounds on this distance in terms of different measures of the polynomial (e.g.\ its height). The simplest one is derived by Liouville's inequality, but then we improve this result and give different bounds using our special determinants.
\end{abstract}

\section{Introduction} \label{sec:intro}

Throughout this paper, $f(x)$ denotes a polynomial of degree $n$ with coefficients $a_i \in \mathbb{C}$ and roots $\alpha_i \in \mathbb{C}$:
\begin{align*}
	f(x) &= a_n x^n + \ldots + a_1 x + a_0 = \\
	&= a_n (x - \alpha_1) (x - \alpha_2) \ldots (x - \alpha_n)
\end{align*}
with $n \geq 1$ and $a_n \neq 0$.

We mainly focus on polynomials with integer coefficients, but some results are stated more generally.

Our main interest is in the following type of polynomials:
\begin{definition} \label{def:expansive}
The polynomial $f$ is \emph{expansive} if all of its roots lie outside the unit circle, i.e.\ $\forall \alpha_i: |\alpha_i| > 1$.
\end{definition}

Expansive polynomials (or matrices with expansive characteristic polynomial) often arise in a wide range of problems involving convergence, where the convergence is ensured by the expansivity of an operator (or rather the contractivity of its inverse). One of our main motivations comes from so-called matrix-based numeration systems \cite[Def.~3.7]{dyndir} \cite{kovacs-binsys}. These systems form a matrix and vector-based generalization of ordinary numeral systems. Many relevant properties of these systems, e.g.\ if all vectors have a representation, the representation is unique, etc., depend on the expansivity of the matrix. The running time of some related algorithms also depends on how tightly the matrix fulfills the condition of expansivity.

In this paper we will use the following functions to measure the complexity of polynomials. Several results are stated in terms of these quantities.
\begin{definition} \label{def:measures}
Denote by $H(f)$, $L(f)$ and $M(f)$ the height, length and Mahler measure of the polynomial $f$, respectively:
\begin{align*}
	H(f) &:= \max_{i=0}^n |a_i|, \\
	L(f) &:= \sum_{i=0}^n |a_i|, \\
	M(f) &:= |a_n| \prod_{i=1}^{n} \max(1, |\alpha_i|).
\end{align*}
\end{definition}
Note that for an expansive polynomial, we simply have $M(f) = |a_0|$.

The paper is built up as follows.
Section~\ref{sec:former} reviews some former results about expansive polynomials.
In Section~\ref{sec:dets}, we introduce a determinant-based condition for the expansivity of real polynomials, and derive some properties of these determinants and their generalizations.
In Section~\ref{sec:bounds}, we give lower bounds on the distance between roots of integer expansive polynomials and the unit circle.
Finally, in Section~\ref{sec:final}, further research directions are given.

\section{Former results} \label{sec:former}

We give several connections between the coefficients of expansive polynomials. First of all, the most trivial one is (since $a_0 = \pm a_n \alpha_1 \alpha_2 \ldots \alpha_n$):
\begin{lemma} \label{thm:exp-cons-lead}
For an expansive polynomial $f$:
\begin{equation*}
	|a_n| < |a_0|.
\end{equation*}
\end{lemma}
The inner coefficients can be bounded using the constant and leading coefficients (proved in \cite{burcsi-phd}):
\begin{lemma} \label{thm:exp-coeffs} 
If $f$ is expansive, then the coefficients have the following bounds:
\begin{equation*}
	|a_k| < \binom{n-1}{k-1} |a_n| + \binom{n-1}{k} |a_0| \quad (1 \leq k \leq n-1),
\end{equation*}
and this is the best possible bound generally.
\end{lemma}
Better bounds can be given if more coefficients are involved: \cite{burcsi-phd} gives bounds on $a_k$ and $a_{n-k}$ in terms of $a_0, a_1, \ldots, a_{k-1}$ and $a_{n-k+1}, \ldots, a_n$ ($k \leq n/2$). This gives a way to exhaustively search for all expansive polynomials with a given degree and constant term by generating the coefficients from outside towards the center.

The conditions above are only necessary but not sufficient. A natural question is how the expansivity of a polynomial can be decided algorithmically. A well-known method for this is the Schur--Cohn test \cite{schur-cohn}, which uses the following transformation:

\begin{definition} \label{def:schur}
The \emph{Schur transform} of $f$ is the polynomial $g$ with:
\begin{gather*}
	g(x) = b_n x^n + \ldots + b_1 x + b_0, \\
	b_k := \overline{a_0} a_k - a_n \overline{a_{n-k}}.
\end{gather*}
\end{definition}
Note that $b_n = 0$, so $\deg g \leq n-1$.
Then we have:
\begin{lemma} \label{thm:schur} \ \\ \vspace{-1\baselineskip}
\begin{itemize}
	\item If $|a_n| < |a_0|$, then $f$ and $g$ have the same number of roots inside the unit circle.
	\item If $|a_n| > |a_0|$, then $f$ has the same number of roots outside the unit circle as $g$ has inside it.
	\item In both cases, $f$ and $g$ share their roots on the unit circle.
\end{itemize}
The roots are counted with multiplicities.
\end{lemma}
It follows that $f$ is expansive if and only if $|a_n| < |a_0|$ and $g$ is expansive.

The algorithm itself (the Schur--Cohn test) works by recursively generating the Schur transform and checking the condition $|a_n| < |a_0|$ for each polynomial in the sequence. Note that due to the decreasing degree, the algorithm terminates with a constant polynomial in at most $n$ iterations.

There are other algorithms for deciding expansivity, for example \cite{burcsi-alg}. It transforms the condition of expansivity to stability, i.e.\ that all roots have negative real parts, by a simple transformation on the polynomial. Then the stability of the resulting polynomial is checked by converting its Hurwitz alternant to continued fraction form, see the details in \cite{burcsi-alg}. The main calculations in this algorithm involve a recurrence relation similar to the Schur transform. According to \cite{burcsi-alg}, this algorithm performs better in practice than the Schur--Cohn test if the polynomial is likely to be expansive.

\section{Special determinants} \label{sec:dets}

\subsection{The $D$-conditions} \label{sec:dets:simple}

\begin{definition} \label{def:dets}
For a polynomial $f$ of degree $n$, define the determinant $D_k^\pm(f)$ for each $1 \leq k \leq n$ and both signs $+$ or $-$ as a function of the coefficients of $f$ as follows. The size of $D_k^\pm(f)$ is $k \times k$, and the element in the $i$th row and $j$th column is the following:
\begin{equation*}
	d_{ij} = a_{j-i} \pm a_{i+j+n-k-1} \quad (1 \leq i,j \leq k)
\end{equation*}
with the convention that indices outside the allowed range indicate zero values, i.e.\ $a_i = 0$ for $i < 0$ and $i > n$.
\end{definition}
For example for $n=7$ and $k=6$:
\begin{equation*}
\scalebox{0.8}{$
\setlength\arraycolsep{5pt}
D_6^-(f) = \begin{vmatrix}
	a_0-a_2 & a_1-a_3 & a_2-a_4 & a_3-a_5 & a_4-a_6 & a_5-a_7 \\
	-a_3 & a_0-a_4 & a_1-a_5 & a_2-a_6 & a_3-a_7 & a_4 \\
	-a_4 & -a_5 & a_0-a_6 & a_1-a_7 & a_2 & a_3 \\
	-a_5 & -a_6 & -a_7 & a_0 & a_1 & a_2 \\
	-a_6 & -a_7 & & & a_0 & a_1 \\
	-a_7 & & & & & a_0
\end{vmatrix}.
$}
\end{equation*}

Then we have the following characterisation of the expansivity of $f$:

\begin{theorem} \label{thm:exp-dets}
Assume that $f$ has real coefficients, i.e.\ all $a_k \in \mathbb{R}$, and $a_0 > 0$. Then:
\begin{enumerate}
	\item $f$ is expansive if and only if for all $k$ between $1 \leq k \leq n$ and for both signs $+$ and $-$: $D_k^\pm(f) > 0$ ($D$-conditions).
	\item The $D$-conditions for $k = n$ can be replaced by the simpler $f(\pm 1) > 0$.
	\item Furthermore, the $D$-conditions are required only for every second $k$, i.e.: \begin{enumerate}
		\item only $k = n, n-2, n-4, \ldots$ are required for statement (1)\ and
		\item only $k = n-1, n-3, n-5, \ldots$ for statement (2).
	\end{enumerate}
\end{enumerate}
\end{theorem}

Note that the assumption $a_0 > 0$ makes no real restriction, since if $a_0 < 0$, the polynomial can be multiplied by $-1$ without changing its roots, and if $a_0 = 0$, the polynomial is not expansive.

\begin{proof}

We use the Schur--Cohn test as described in Section~\ref{sec:former}. Within this proof, we release the assumption that $a_n \neq 0$, so allow $\deg f < n$. Note that Lemma~\ref{thm:schur} about the Schur transform is still (trivially) true in this case (provided that $a_0 \neq 0$, which is assumed in this theorem). This allows a technical simplification of the Schur--Cohn test: we can always reduce $n$ by exactly $1$, even if the degree drops by more. For unambiguity, we indicate $n$ on the $D$-conditions by writing $D_{n,k}^\pm$ instead of $D_k^\pm$ in this proof.

(1) The first statement is proved by induction on $n$. By Lemma~\ref{thm:schur}, the expansivity of $f$ is equivalent to $|a_n| < |a_0|$ and that its Schur transform $g$ is expansive, so in the induction step, we need to prove that
\begin{equation} \label{D-induction}
\begin{aligned}
	&\forall k \in \{1, 2, \ldots, n-1, n\}: D_{n,k}^\pm(f) > 0 \;\Longleftrightarrow\; \\
	\;\Longleftrightarrow\; |a_n| < |a_0| \;\wedge\; &\forall k \in \{1, 2, \ldots, n-1\}: D_{n-1,k}^\pm(g) > 0.
\end{aligned}
\end{equation}

For this, it is sufficient to prove both that
\begin{equation} \label{D-first}
	D_{n,1}^\pm(f) > 0 \Longleftrightarrow |a_n| < |a_0|
\end{equation}
and that for each $k$ between $2 \leq k \leq n$ and for both signs:
\begin{equation} \label{D-ind-part}
	|a_n| < |a_0| \Longrightarrow \left( D_{n,k}^\pm(f) > 0 \Longleftrightarrow D_{n-1,k-1}^\pm(g) > 0 \right).
\end{equation}

Since $D_{n,1}^\pm(f) = a_0 \pm a_n$ and $a_0 > 0$ is assumed, (\ref{D-first}) is trivially true.

For $n=1$, (\ref{D-first}) is the only condition.

Now we prove (\ref{D-ind-part}). Recall from Definition~\ref{def:schur} that the coefficients of $g$ are $b_k = a_0 a_k - a_n a_{n-k}$. Therefore, by Definition~\ref{def:dets}, the elements of the determinant $D_{n-1,k-1}^\pm(g)$ are (for $1 \leq i,j \leq k-1$):
\begin{equation} \label{g-dij}
\begin{aligned}
	d_{ij} &= b_{j-i} \pm b_{i+j+n-k-1} = \\
	&= a_0 (a_{j-i} \pm a_{n+i+j-k-1}) - a_n (a_{n+i-j} \pm a_{-i-j+k+1}).
\end{aligned}
\end{equation}
We claim that this has the same determinant as the following $(2k-2) \times (2k-2)$ matrix ($D^{(0)}$):
\begin{equation} \label{d0-ij}
d_{ij}^{(0)} = \begin{cases}
	a_{j-i} \pm a_{n+i+j-k-1} & (1 \leq i, j \leq k-1), \\
	a_{-i-j+2k} \pm a_{n+i-j+k-1} & (1 \leq i \leq k-1,\; k \leq j \leq 2k-2), \\
	\pm a_n \delta_{i-k+1,j} & (k \leq i \leq 2k-2,\; 1 \leq j \leq k-1), \\
	a_0 \delta_{i,j} & (k \leq i, j \leq 2k-2).
\end{cases}
\end{equation}
We can check this by clearing the bottom-left quarter of $D^{(0)}$ using column-transformations: add $\mp a_n/a_0$ times the columns $k \leq j \leq 2k-2$ to the columns $1 \leq j \leq k-1$ respectively, then it is easy to check that the top left quarter is (\ref{g-dij}), but multiplied by $1/a_0$, the bottom left quarter is zero, and the bottom right quarter is a diagonal matrix with all $a_0$, therefore the determinant indeed equals $D_{n-1,k-1}^\pm(g)$.

Now notice that $d_{i,j}^{(0)} = d_{i,2k-j}^{(0)}$ for $i \leq k$ and $j \geq 2$, so we can subtract the $j$th column from the $(2k-j)$th one for $2 \leq j \leq k-1$ to get $D^{(1)}$:
\begin{equation*}
d^{(1)}_{ij} = \begin{cases}
	a_{j-i} \pm a_{n+i+j-k-1} & (1 \leq i, j \leq k), \\
	0 & (1 \leq i \leq k,\; k+1 \leq j \leq 2k-2), \\
	\pm a_n \delta_{i-k+1,j} & (k+1 \leq i \leq 2k-2,\; 1 \leq j \leq k), \\
	a_0 \delta_{i,j} \mp a_n \delta_{i-k+1,2k-j} & (k+1 \leq i, j \leq 2k-2).
\end{cases}
\end{equation*}
Its determinant is the product of the top left $k \times k$ minor, which is exactly $D_{n,k}^\pm(f)$, and the bottom right $(k-2) \times (k-2)$ minor, which is the following:
\begin{equation*}
\scalebox{0.75}{$
\begin{vmatrix}
	a_0 & & & & \mp a_n \\
	& a_0 & & \mp a_n & \\
	& & \ddots & & \\
	& \mp a_n & & a_0 & \\
	\mp a_n & & & & a_0
\end{vmatrix}_{(k-2) \times (k-2)}
=
\begin{cases}
	\left( a_0^2 - a_n^2 \right)^{\frac{k-2}{2}} & (2 \mid k), \\
	\left( a_0^2 - a_n^2 \right)^{\frac{k-3}{2}} (a_0 \mp a_n) & (2 \nmid k).
\end{cases}
$}
\end{equation*}
This determinant is positive, because $|a_n| < |a_0|$ is assumed, so we can conclude that the original determinant ($D_{n-1,k-1}^\pm(g)$) and $D_{n,k}^\pm(f)$ has the same sign. This proves (\ref{D-ind-part}), which finishes the first part of the proof.

(2) For the second statement, we prove that $D_{n,n}^\pm(f) = f(\pm 1) D_{n,n-1}^-(f)$. Then, since $D_{n,n-1}^-(f) > 0$ is one of the remaining $D$-conditions, the other two expressions have the same sign. We perform a similarity transformation on the matrix of $D_{n,n}^\pm(f)$ to get another matrix with the same determinant, from which the desired factorization is obvious. First we illustrate it by an example, which shows that $D_{4,4}^-(f) = f(-1) D_{4,3}^-(f)$:
\begin{equation*}
\setlength\arraycolsep{8pt}
\scalebox{0.8}{$
\begin{bmatrix}
	1 & -1 & 1 & -1 \\
	0 & 1 & -1 & 1 \\
	0 & 0 & 1 & -1 \\
	0 & 0 & 0 & 1
\end{bmatrix} \cdot
\begin{bmatrix}
	a_0-a_1 & a_1-a_2 & a_2-a_3 & a_3-a_4 \\
	-a_2 & a_0-a_3 & a_1-a_4 & a_2 \\
	-a_3 & -a_4 & a_0 & a_1 \\
	-a_4 & & & a_0
\end{bmatrix} \cdot
\begin{bmatrix}
	1 & 1 & 0 & 0 \\
	0 & 1 & 1 & 0 \\
	0 & 0 & 1 & 1 \\
	0 & 0 & 0 & 1
\end{bmatrix} =
$}
\end{equation*}
\begin{equation*}
\setlength\arraycolsep{8pt}
\scalebox{0.8}{$
= \begin{bmatrix}
	a_0-a_1+a_2-a_3+a_4 & & & \\
	-a_2+a_3-a_4 & a_0-a_2 & a_1-a_3 & a_2-a_4 \\
	a_4-a_3 & -a_3 & a_0-a_4 & a_1 \\
	-a_4 & -a_4 & & a_0
\end{bmatrix}.
$}
\end{equation*}

Generally, the similarity matrix on the right is an upper bidiagonal matrix whose main diagonal entries are $1$, and the diagonal above it contains $\mp 1$ (i.e.\ the opposite sign than that of $D_{n,n}^\pm(f)$). Its inverse (on the left) is an upper triangular matrix whose entries are $t_{i,j} = (\pm 1)^{j-i}$ for $i \leq j$. The entries of the transformed matrix (in the middle) are, by Definition~\ref{def:dets}, $d_{i,j} = a_{j-i} \pm a_{i+j-1}$. Then, the top left entry of the resulting matrix is indeed $f(\pm 1)$:
\begin{align*}
	d'_{1,1} = \sum_{l=1}^{n} (\pm 1)^{l-1} d_{l,1}
	= a_0 + \sum_{l=1}^{n} (\pm 1)^l a_l = f(\pm 1).
\end{align*}

We can also prove that the resulting matrix contains $D_{n,n-1}^-(f)$ in the bottom right corner, by calculating the elements from the second column ($j \geq 2$):
\begin{align*}
	d'_{i,j} &= \sum_{l=i}^{n} (\pm 1)^{l-i} (d_{l,j} \mp d_{l,j-1}) = \\
	&= \sum_{l=i}^{n} ( (\pm 1)^{l-i} (a_{j-l} - a_{l+j-2}) - (\pm 1)^{l-i+1} (a_{j-l-1} - a_{l+j-1}) ) = \\
	&= (\pm 1)^{i-i} (a_{j-i} - a_{i+j-2}) - (\pm 1)^{n-i+1} (a_{j-n-1} - a_{n+j-1}) = \\
	&= a_{j-i} - a_{i+j-2},
\end{align*}
which is indeed the appropriate element of $D_{n,n-1}^-(f)$ for $i \geq 2$, and it also shows that the first row is $0$ for $j \geq 2$. By the rules of determinants, it follows that $D_{n,n}^\pm(f) = f(\pm 1) D_{n,n-1}^-(f)$.

(3) For the third statement, we need to prove a modified version of (\ref{D-induction}), using (\ref{D-first}) and (\ref{D-ind-part}). We replace the sets $\{1, \ldots, n\}$ and $\{1, \ldots, n-1\}$ in (\ref{D-induction}) by some subsets, say $K$ and $K'$ respectively.

There are four cases, depending on the parity of $n$ and the two parts of the statement:
\begin{align}
	\text{(a)},\; 2 \mid n: \quad
		&K = \{2, 4, \ldots, n-2, n\},\; K' = \{1, 3, \ldots, n-3, n-1\}, \\
	\text{(a)},\; 2 \nmid n: \quad
		&K = \{1, 3, \ldots, n-2, n\},\; K' = \{2, 4, \ldots, n-3, n-1\}, \\
	\text{(b)},\; 2 \mid n: \quad
		&K = \{1, 3, \ldots, n-1, n\},\; K' = \{2, 4, \ldots, n-2, n-1\}, \\
	\text{(b)},\; 2 \nmid n: \quad
		&K = \{2, 4, \ldots, n-1, n\},\; K' = \{1, 3, \ldots, n-2, n-1\}.
\end{align}

Note that for part (b), we included $k = n$ so that we can replace it by $f(\pm 1) > 0$ in the same way as in the proof of the second statement (2).

Now the $\Longleftarrow$ direction of the modified (\ref{D-induction}) is true in all four cases, because it is sufficient that $\forall k \geq 2: k \in K \Longleftrightarrow k-1 \in K'$ (due to (\ref{D-first}) and (\ref{D-ind-part})). The $\Longrightarrow$ direction however does not work unless $1 \in K$, which is only true in two cases. In the other two, we prove that the $k=2$ case, i.e.\ $D_{n,2}^\pm(f) > 0$, implies the $k=1$ case, $D_{n,1}^\pm(f) > 0$:
\begin{align*}
	D_{n,2}^\pm(f) &= a_0(a_0 \pm a_{n-1}) - a_n(a_n \pm a_1) > 0 \;\Longleftrightarrow \\
	&\Longleftrightarrow\; a_0^2 - a_n^2 > |a_0 a_{n-1} - a_n a_1| \;\Longrightarrow \\
	&\Longrightarrow\; a_0^2 - a_n^2 > 0
	\;\Longleftrightarrow\; |a_n| < |a_0|
	\;\Longleftrightarrow\; D_{n,1}^\pm(f) > 0.
\end{align*}
\end{proof}

We emphasize again that this theorem works for real coefficients only. A crucial use of this assumption is in the proof of the equivalence (\ref{D-first}), when we converted the condition $|a_n| < a_0$ to $a_0 \pm a_n > 0$.

\subsection{Complexity of the algorithm} \label{sec:exp-time}

Theorem~\ref{thm:exp-dets} can be turned into an algorithm that decides the expansivity of a polynomial, by calculating the determinants $D_k^\pm(f)$ and checking their sign. Normally, with numerical calculation (i.e.\ with standard floating-point numbers) the Schur--Cohn test is faster. The latter however has problems when the coefficients are integers and exact calculation is performed (i.e.\ with multi-precision arithmetic). Then, the Schur transform $b_k := a_0 a_k - a_n a_{n-k}$ may double the length of the coefficients in each step, which may lead to exponential coefficient growth and therefore running time. On the contrary, we show that our algorithm with the $D$-conditions runs in polynomial time.

We use the Bareiss algorithm \cite{bareiss} to calculate the determinants involving integer coefficients. It is an exact (not numerical) algorithm based on Gaussian elimination, which does not allow coefficient explosion (similar to the Schur--Cohn test) nor perform costly gcd-calculations to simplify exact rational entries, but instead it performs smart simplifications to ensure moderate (polynomial) growth of the entries and therefore polynomial running time. More specifically, its running time for an $n \times n$ matrix with integer entries is $O( n^5 \log(nB)^2 )$, where each element is $|a_{ij}| \leq B$.

For deciding expansivity of a polynomial, the Bareiss algorithm is performed on $D_1^\pm(f), D_2^\pm(f), \ldots, D_{n-1}^\pm(f)$ as the second statement of Theorem~\ref{thm:exp-dets} ensures (and also $f(\pm1)$ is calculated, but it is negligible). The entries of the matrices are at most $2 H$ in absolute value (where $H := H(f)$, the height of $f$). The running time of the algorithm is therefore:
\begin{equation*}
	T_{D\text{-exp}}(n, H) = O\!\left( n^6 \log(n H)^2 \right).
\end{equation*}

This is the worst-case complexity, and it is realized when every condition needs to be checked, e.g.\ when the polynomial is expansive. On the contrary, when an easy condition fails, the check terminates quickly. It is therefore advisable to start with the simpliest conditions $D_1^\pm(f) > 0$, which are simply $|a_n| < |a_0|$, then continue with $f(\pm1) > 0$, then continue with larger and larger $D$-conditions. And also, due to the third statement of Theorem~\ref{thm:exp-dets}, every other $D$-condition may be skipped (i.e.\ $k = \ldots, n-5, n-3, n-1$ is sufficient), but still, for the first few $k$, it is suggested to check them all regardless of parity, because they may give a quick negative result.

\subsection{The $D$-polynomials} \label{sec:dets:poly}

We can extend Theorem~\ref{thm:exp-dets} from expansivity to the property that all roots have greater absolute value than a given constant, by replacing the expressions $D_k^\pm(f)$ by polynomials:

\begin{definition} \label{def:d-poly}
Define $\widetilde{D}_k^\pm(f)(x)$ as $D_k^\pm(f)$ in Definition~\ref{def:dets} except that all $a_j$ are replaced by $a_j x^j$.
\end{definition}
For example for $n=7$ and $k=6$:
\begin{equation*}
\scalebox{0.8}{$
\setlength\arraycolsep{5pt}
\widetilde{D}_6^-(f)(x) = \begin{vmatrix}
	a_0-a_2 x^2 & a_1 x-a_3 x^3 & a_2 x^2-a_4 x^4 & a_3 x^3-a_5 x^5 & a_4 x^4-a_6 x^6 & a_5 x^5-a_7 x^7 \\
	-a_3 x^3 & a_0-a_4 x^4 & a_1 x-a_5 x^5 & a_2 x^2-a_6 x^6 & a_3 x^3-a_7 x^7 & a_4 x^4 \\
	-a_4 x^4 & -a_5 x^5 & a_0-a_6 x^6 & a_1 x-a_7 x^7 & a_2 x^2 & a_3 x^3 \\
	-a_5 x^5 & -a_6 x^6 & -a_7 x^7 & a_0 & a_1 x & a_2 x^2 \\
	-a_6 x^6 & -a_7 x^7 & & & a_0 & a_1 x \\
	-a_7 x^7 & & & & & a_0
\end{vmatrix}.
$}
\end{equation*}

Note that there is a direct connection between $D$ and $\widetilde{D}$: first, $\widetilde{D}_k^\pm(f)(1) = D_k^\pm(f)$, and in the other direction, let $f_y(x) := f(xy)$, then $D_k^\pm(f_y) = \widetilde{D}_k^\pm(f)(y)$.

Then we have:
\begin{corollary}
If $f$ has real coefficients with $a_0 > 0$, then for any $s > 0$, $f$ has all its roots $|\alpha_i| > s$ if and only if similar conditions hold as in Theorem~\ref{thm:exp-dets}, but the expressions $D_k^\pm(f)$ are replaced by $\widetilde{D}_k^\pm(f)(s)$, and $f(\pm 1)$ is replaced by $f(\pm s)$.
\end{corollary}
\begin{proof}
Apply Theorem~\ref{thm:exp-dets} for $f_s(x) := f(sx)$, then the condition $|\alpha_i| > s$ for the roots of $f$ is equivalent to expansivity of $f_s$.
\end{proof}

Some properties of the $D$-polynomials:
\begin{lemma} \label{thm:d-polys} \ \\ \vspace{-1\baselineskip}
\begin{enumerate}
	\item $\widetilde{D}_k^\pm(f)(x)$ has constant term $a_0^k$ and leading term $a_n^k x^{k n}$, the latter having a sign dependent on the parameters ($k$ and $\pm$).
	\item For $k = n-1, n-3, n-5, \ldots$, $\widetilde{D}_k^\pm(f)(x)$ has only even powers of $x$.
	\item The polynomial $\widetilde{D}_{n-1}^-(f)(x^{1/2})$ has the following $\binom{n}{2}$ roots: $\alpha_i \alpha_j$ for all $i < j$.
\end{enumerate}
\end{lemma}
\begin{proof}
(1) Each row of the $k \times k$ determinant $\widetilde{D}_k^\pm(f)(x)$ has a unique term with miminal and maximal $x$-power, namely $a_0$ and $\pm a_n x^n$. The formers are in the main diagonal, yielding the constant term $a_0^k$, and the latters are in the antidiagonal, yielding $(-1)^{\lfloor k/2 \rfloor} (\pm 1)^k a_n^k x^{k n}$ as the leading term.

(2) By Definition~\ref{def:dets}, the determinant $\widetilde{D}_k^\pm(f)(x)$ has entries $d_{ij} = a_{j-i} x^{j-i} \pm a_{i+j+n-k-1} x^{i+j+n-k-1}$. Multiplying the $i$th row by $x^{i-1}$ and dividing the $j$th column by $x^{j-1}$ gives an other representation of the same determinant: $d'_{ij} = a_{j-i} \pm a_{i+j+n-k-1} x^{2i+n-k-1}$, from which it is obvious that if $n-k-1$ is even, it has only even powers of $x$. (Note that for $x=0$, the transformation does not work, but then the two determinants are trivially equal.)

(3) We construct an other polynomial of degree $n^2$ whose roots are $\alpha_i \alpha_j$ for all $1 \leq i,j \leq n$ (it easily follows from the rules of resultants) \cite[p.~159]{cohen}:
\begin{equation*}
	F(x) = \operatorname{res}_y\! \left( f(y), f\left(\frac{x}{y}\right) y^n \right) = a_n^{2n} \prod_{i=1}^{n} \prod_{j=1}^{n} (x - \alpha_i \alpha_j).
\end{equation*}
Up to sign, it can be represented by the following $2n \times 2n$ determinant, which is a slightly modified version of the Sylvester matrix corresponding to the resultant above:
\begin{equation*}
	(-1)^n F(x) =
	\begin{vmatrix}
		a_0 & & & & a_n & & & \\
		a_1 x & a_0 & & & a_{n-1} & a_n & & \\
		a_2 x^2 & a_1 x & \ddots & & \vdots & a_{n-1} & \ddots & \\
		\vdots & a_2 x^2 & \ddots & a_0 & a_1 & \vdots & \ddots & a_n \\
		a_n x^n & \vdots & \ddots & a_1 x & a_0 & a_1 &        & a_{n-1} \\
		& a_n x^n & & a_2 x^2 & & a_0 & \ddots & \vdots \\
		& & \ddots & \vdots & & & \ddots & a_1 \\
		& & & a_n x^n & & & & a_0
	\end{vmatrix}.
\end{equation*}

We will factor $F(x^2)$ by applying a similarity transformation on this matrix but with $x^2$ instead of $x$. First we illustrate the idea for $n = 3$: \vspace{-1.5ex}

\begin{equation*}
\scalebox{0.8}{$\begin{gathered}
	\begin{bmatrix}
		x^2 & & & & & \\
		& x & & & & \\
		& & 1 & & & \\
		& & x & 1 & & \\
		& x^2 & & & x^{-1} & \\
		x^3 & & & & & x^{-2}
	\end{bmatrix} \cdot
	\begin{bmatrix}
		a_0 & & & a_3 & & \\
		a_1 x^2 & a_0 & & a_2 & a_3 & \\
		a_2 x^4 & a_1 x^2 & a_0 & a_1 & a_2 & a_3 \\
		a_3 x^6 & a_2 x^4 & a_1 x^2 & a_0 & a_1 & a_2 \\
		& a_3 x^6 & a_2 x^4 & & a_0 & a_1 \\
		& & a_3 x^6 & & & a_0
	\end{bmatrix} \cdot
	\begin{bmatrix}
		x^{-2} & & & & &\\
		& x^{-1} & & & & \\
		& & 1 & & & \\
		& & -x & 1 & & \\
		& -x^2 & & & x & \\
		-x^3 & & & & & x^2
	\end{bmatrix} = \\
	= \begin{bmatrix}
		a_0 & & -a_3 x^3 & a_3 x^2 & & \\
		a_1 x & a_0 - a_3 x^3 & -a_2 x^2 & a_2 x & a_3 x^2 & \\
		a_2 x^2 - a_3 x^3 & a_1 x - a_2 x^2 & a_0 - a_1 x & a_1 & a_2 x & a_3 x^2 \\
		& & & a_1 x + a_0 & a_2 x^2 + a_1 x & a_3 x^3 + a_2 x^2 \\
		& & & a_2 x^2 & a_3 x^3 + a_0 & a_1 x \\
		& & & a_3 x^3 & & a_0
	\end{bmatrix}.
\end{gathered}$}
\end{equation*}

Generally, calling this as $U M V = N$ with $U = V^{-1}$, the entries are:
\begin{equation*}
	m_{ij} = \begin{cases}
		a_{i-j} x^{2i-2j}   & (1 \leq j \leq n), \\
		a_{j-i}             & (n+1 \leq j \leq 2n),
	\end{cases}
\end{equation*}
\begin{equation*}
	u_{ij}; v_{ij} = \begin{cases}
		x^{n-j}; x^{j-n} & (1 \leq i = j \leq n), \\
		x^{n-j+1}; x^{j-n-1} & (n+1 \leq i = j \leq 2n), \\
		x^{n-j+1}; -x^{n-j+1} & (i = 2n-j+1, 1 \leq j \leq n), \\
		0 & (\text{otherwise}).
	\end{cases}
\end{equation*}

Then straightforward calculation gives:
\begin{equation*}
	(U M)_{ij} = \begin{cases}
		a_{i-j} x^{n+i-2j} & (1 \leq i,j \leq n), \\
		a_{j-i} x^{n-i} & (1 \leq i \leq n, n+1 \leq j \leq 2n), \\
		a_{i-j} x^{n+i-2j+1} + a_{2n-i-j+1} x^{3n-i-2j+2} & (n+1 \leq i \leq 2n, 1 \leq j \leq n), \\
		a_{j-i} x^{n-i+1} + a_{i+j-2n-1} x^{i-n} & (n+1 \leq i,j \leq 2n),
	\end{cases}
\end{equation*}
\begin{equation*}
	(U M V)_{ij} = \begin{cases}
		a_{i-j} x^{i-j} - a_{2n-i-j+1} x^{2n-i-j+1} & (1 \leq i,j \leq n), \\
		a_{j-i} x^{j-i-1} & (1 \leq i \leq n, n+1 \leq j \leq 2n), \\
		0 & (n+1 \leq i \leq 2n, 1 \leq j \leq n), \\
		a_{j-i} x^{j-i} + a_{i+j-2n-1} x^{i+j-2n-1} & (n+1 \leq i,j \leq 2n).
	\end{cases}
\end{equation*}

This confirms what the example suggests about the resulting matrix, $N$: the bottom right quadrant is the exact copy of $\widetilde{D}_n^+(f)(x)$, the top left quadrant is $\widetilde{D}_n^-(f)(x)$ turned upside down, and the bottom left quadrant is zero. (Note that again for $x=0$, the transformation fails, but then trivially $\det M = \det N$.) This proves that $F(x^2) = (-1)^n \widetilde{D}_n^+(f)(x) \cdot \widetilde{D}_n^-(f)(x)$.

In the proof of Theorem~\ref{thm:exp-dets} (2), we found that $D_n^\pm(f) = f(\pm 1) D_{n-1}^-(f)$. Using the connection between $D$ and $\widetilde{D}$, this factorization can be generalized to $D$-polynomials, which gives a finer factorization of $F(x^2)$:
\begin{equation*}
	F(x^2) = (-1)^n f(x) f(-x) \left( \widetilde{D}_{n-1}^-(f)(x) \right)^2.
\end{equation*}
Examine the first few factors:
\begin{equation*}
	(-1)^n f(x) f(-x) = a_n \prod_{i=1}^{n} (x - \alpha_i) \cdot a_n \prod_{i=1}^{n} (x + \alpha_i)
	= a_n^2 \prod_{i=1}^{n} (x^2 - \alpha_i^2).
\end{equation*}
This shows that the polynomial $f(x^{1/2}) f(-x^{1/2})$ has all $\alpha_i^2$ as roots. All other roots of $F(x)$ are double ($\alpha_i \alpha_j = \alpha_j \alpha_i$ with $i \neq j$), therefore $\widetilde{D}_{n-1}^-(f)(x^{1/2})$ has half of them, i.e.\ all $\alpha_i \alpha_j$ with $i < j$.
\end{proof}

\section{Bounds on the expansivity gap} \label{sec:bounds}

In this section we examine the following question: given an expansive polynomial with \emph{integer} coefficients, how close can the size of the roots be to $1$?
\begin{definition}
Let $f$ be an expansive polynomial (i.e.\ with roots $|\alpha_i| > 1$), then the \emph{expansivity gap} is:
\begin{equation*}
	\varepsilon := \min_{i=1}^{n} |\alpha_i| - 1.
\end{equation*}
\end{definition}
Our goal is to give a lower bound on $\varepsilon$ in terms of the degree $n$ and some other quantity measuring the complexity of $f$. For example, we can use the height and the length of the polynomial, as defined in Definition~\ref{def:measures}.

The expansivity gap is closely related to the distance of algebraic numbers (i.e.\ roots of integer polynomials) from $1$. If the root $\alpha$ of an expansive polynomial is real, then $|\alpha| - 1$ is either $\alpha - 1$ or $(-\alpha) - 1$, otherwise $|\alpha| = \sqrt{\alpha \overline\alpha}$, so when $\alpha$ is close to 1, $|\alpha| - 1 \approx \frac{\alpha \overline\alpha - 1}{2}$.

\subsection{Liouville-type bounds} \label{sec:bounds:liou}

A common tool for bounding distances between algebraic numbers is Liouville's inequality \cite[Prop.~3.14]{waldschmidt}:
\begin{theorem} \label{thm:liou}
If $f$ and $g$ are integer polynomials and $\alpha$ is a root of $f$ but not of $g$, the following bound holds:
\begin{equation*}
	\frac{1}{|g(\alpha)|} \leq L(g)^{n-1} M(f)^{\deg g},
\end{equation*}
where $L(g)$ is the length and $M(f)$ is the Mahler measure as defined in Definition~\ref{def:measures}.
\end{theorem}

Using Liouville's inequality, we give the following bounds on the expansivity gap $\varepsilon$, or rather on $1 / \varepsilon$ in order to keep the formulas simpler:
\begin{theorem} \label{thm:bound-A}
For an expansive polynomial $f$ with integer coefficients, and any root $\alpha$ of $f$, we have:
\begin{equation*}
	\frac{1}{|\alpha| - 1} \leq \begin{cases}
		2^{n-1} |a_0| & (\alpha \in \mathbb{R}), \\
		2^{\binom{n}{2}} |a_0|^{n-1} + 1 & (\alpha \in \mathbb{C} \setminus \mathbb{R}).
	\end{cases}
\end{equation*}
\end{theorem}

\begin{proof}
For real $\alpha$, applying Liouville's inequality for $g(x) = x + 1$ and $g(x) = x - 1$ and using that $M(f) = |a_0|$ for expansive polynomials, we get:
\begin{equation} \label{liou-one}
	\frac{1}{|\alpha \pm 1|} \leq 2^{n-1} |a_0|.
\end{equation}

For non-real $\alpha$, we apply this inequality to $\alpha \overline\alpha$. Since $\overline\alpha$ is also a root of $f$, $\alpha \overline\alpha$ is a root of the polynomial $F(x) := \widetilde{D}_{n-1}^-(f)(x^{1/2})$, because it has all $\alpha_i \alpha_j$ as roots for $i < j$ by Lemma~\ref{thm:d-polys}. As $F$ is also an expansive polynomial, and it has degree $\binom{n}{2}$ and constant term is $a_0^{n-1}$, we get from (\ref{liou-one}):
\begin{equation*}
	\frac{1}{|\alpha \overline\alpha - 1|} \leq 2^{\binom{n}{2} - 1} |a_0|^{n-1}.
\end{equation*}
We can rearrange this as $|\alpha|^2 = \alpha \overline\alpha \geq 1 + \frac{1}{B}$ where $B$ is the right-hand side of the inequality above. The proof finishes by applying the fact that
\begin{equation} \label{sqrt-aux}
	\sqrt{1 + \frac{1}{B}} > 1 + \frac{1}{2B + 1}
\end{equation}
for any $B > 0$.
\end{proof}

We are able to give slightly stronger bounds by using not only the constant term, but also the leading coefficient:
\begin{theorem} \label{thm:bound-AZ}
For any root $\alpha$ of an expansive integer polynomial $f$:
\begin{equation*}
	\frac{1}{|\alpha| - 1} \leq \begin{cases}
		2^{n-2} (|a_0| + |a_n|) & (\alpha \in \mathbb{R}), \\
		2^{\binom{n-1}{2}} (|a_0| + |a_n|)^{n-1} + 1 & (\alpha \in \mathbb{C} \setminus \mathbb{R}).
	\end{cases}
\end{equation*}
\end{theorem}

\begin{proof}
We give a better alternative to Liouville's inequality in its special form (\ref{liou-one}). For this, first we present a simple proof of (\ref{liou-one}). We start by dividing the factorized form $f(x) = a_n (x - \alpha_1) (x - \alpha_2) \ldots (x - \alpha_n)$ by one of its root factors, then applying this for $|f(\pm1)|$:
\begin{gather*}
	\frac{|f(\pm1)|}{|\pm\!1 - \alpha_1|}
	= |a_n| \prod_{i=2}^{n} |\pm\!1 - \alpha_i|
	\leq |a_n| \prod_{i=2}^{n} (1 + |\alpha_i|) \leq \\
	\leq \frac{1}{2} |a_n| \prod_{i=1}^{n} (1 + |\alpha_i|)
	\leq \frac{1}{2} |a_n| \prod_{i=1}^{n} 2 |\alpha_i|
	= 2^{n-1} |a_0|.
\end{gather*}
Since $f$ is expansive and has integer coefficients, $|f(\pm1)| \geq 1$, which finishes the proof of (\ref{liou-one}).

We will modify this proof as follows. Define the polynomial $\hat{f}$ to have $-|\alpha_i|$ as roots, more precisely:
\begin{equation*}
	\hat{f}(x) = \hat{a}_n x^n + \ldots + \hat{a}_1 x + \hat{a}_0
	:= |a_n| (x + |\alpha_1|) (x + |\alpha_2|) \ldots (x + |\alpha_n|).
\end{equation*}
Now we can rewrite one consequence of the proof above as $\frac{1}{|\alpha \pm 1|} \leq \frac{1}{2} \hat{f}(1)$, and the next task is to give a better bound on $\hat{f}(1)$.

For this, we prove that for any expansive polynomial $f$, the following holds:
\begin{equation} \label{expansive-aux}
	|f(x)| \leq (1 + |x|)^{n-1} (|a_n| |x| + |a_0|).
\end{equation}
Indeed, by using the the inequalities (\ref{thm:exp-coeffs}) between the coefficients of expansive polynomials:
\begin{align*}
	|f(x)| &\leq \sum_{k=0}^{n} |a_k| |x|^k
	\leq \sum_{k=0}^{n} \left( \binom{n-1}{k-1} |a_n| + \binom{n-1}{k} |a_0| \right) |x|^k = \\
	&= \sum_{k=0}^{n-1} \binom{n-1}{k} |x|^k ( |a_n| |x| + |a_0| ),
\end{align*}
and the binomial theorem gives (\ref{expansive-aux}).

As $\hat{f}$ is also expansive, we can apply (\ref{expansive-aux}) to $\hat{f}(1)$, which completes the proof of the theorem for real $\alpha$.

For non-real $\alpha$, we apply the same method but for $\alpha \overline\alpha$ and $F$, and define $\hat{F}$ like $\hat{f}$ but for $F$. Now $\hat{F}(1)$ can be written as:
\begin{equation*}
	\hat{F}(1)
	= \prod_{m=2}^{n} \left( |a_n| \prod_{i=1}^{m-1} (1 + |\alpha_i| |\alpha_m|) \right)
	= \prod_{m=2}^{n} \left( |\alpha_m|^{m-1} \hat{f}_{m-1} \left( \frac{1}{|\alpha_m|} \right) \right)
\end{equation*}
with $\hat{f}_m(x) := |a_n| (x + |\alpha_1|) \ldots (x + |\alpha_m|)$. This is an expansive polynomial, so we can use (\ref{expansive-aux}) and continue:
\begin{align*}
	&\leq \prod_{m=2}^{n} \left( (1 + |\alpha_m|)^{m-2} |a_n| (1 + |\alpha_1| \ldots |\alpha_m|) \right) = \\
	&= \prod_{l=2}^{n} |a_n| (1 + |\alpha_1| \ldots |\alpha_l|) (1 + |\alpha_{l+1}|) \ldots (1 + |\alpha_n|)
\end{align*}
If we treat the product inside $\prod$ as an expansive polynomial with $x=1$, we can apply (\ref{expansive-aux}) to it:
\begin{equation*}
	\leq \prod_{l=2}^{n} 2^{n-l} |a_n| (1 + |\alpha_1| \ldots |\alpha_n|)
	= 2^{\binom{n-1}{2}} (|a_n| + |a_0|)^{n-1}.
\end{equation*}

We have $\frac{1}{\alpha \overline\alpha - 1} \leq \frac{1}{2} \hat{F}(1)$ like in the first part of the proof, and we can finish the proof as in Theorem~\ref{thm:bound-A}.
\end{proof}

\subsection{Bounds using the determinants} \label{sec:bounds:det}

The previous bounds (Theorem~\ref{thm:bound-A} and Theorem~\ref{thm:bound-AZ}) relied on Liouville's inequality and basic properties of expansive polynomials. They used only the size of the constant term and possibly the leading coefficient, but with quite large $n$-dependent factor. In this section we will use the determinant structure of the $D$-polynomials to give bounds with smaller factors but with all coefficients of the polynomial. We present two theorems, one using the height of the polynomial ($H(f)$) and one using the length ($L(f)$).

\begin{theorem} \label{thm:bound-H}
For any root $\alpha$ of an expansive integer polynomial $f$:
\begin{equation*}
	\frac{1}{|\alpha| - 1} \leq \begin{cases}
		\binom{n+1}{2} H(f) + \frac{n}{2} & (\alpha \in \mathbb{R}), \\
		\binom{n}{2} n! H(f)^{n-1} + \binom{n}{2} + 1 & (\alpha \in \mathbb{C} \setminus \mathbb{R}).
	\end{cases}
\end{equation*}
\end{theorem}

\begin{proof}
Since $|f(\pm1)| \geq 1$, if we can show that $|f(\pm(1 + \varepsilon)) - f(\pm1)| < 1$ for some positive $\varepsilon$, then $f$ has no real root with absolute value $1 + \varepsilon$.

The (finite) power series expansion of $f(1 + \varepsilon)$ is:
\begin{equation*}
	f(1 + \varepsilon) = f(1) + f'(1) \varepsilon + \frac{f''(1)}{2} \varepsilon^2 + \ldots + \frac{f^{(n)}(1)}{n!} \varepsilon^n,
\end{equation*}
and that of $f(-1 - \varepsilon)$ is similar but with some different signs. We can bound the coefficients as follows:
\begin{equation*}
	\left| \frac{f^{(k)}(\pm 1)}{k!} \right|
	= \left| \sum_{j=0}^{n} (\pm 1)^{j-k} \binom{j}{k} a_j \right|
	\leq \binom{n+1}{k+1} H(f).
\end{equation*}
Therefore:
\begin{equation*}
	|f(\pm(1 + \varepsilon)) - f(\pm1)|
	\leq \sum_{k=1}^{n} \binom{n+1}{k+1} \varepsilon^k H(f)
	< \sum_{k=1}^{\infty} (n+1) \left( \frac{n}{2} \varepsilon \right)^k H(f).
\end{equation*}
It is easy to show that for any $A, B > 0$:
\begin{equation} \label{powers-aux}
	\sum_{k=1}^{\infty} \left( A \varepsilon \right)^k B \leq 1
	\quad \Longleftrightarrow \quad
	\varepsilon \leq \frac{1}{A(B+1)}.
\end{equation}
From this, the statement for real $\alpha$ follows.

For non-real $\alpha$, we do a similar calculation but for $F(x)$ (as in the proof of Theorem~\ref{thm:bound-A}). Let $F(x) = A_N x^N + \ldots + A_2 x^2 + A_1 x + A_0$, and recall that $F$ has degree $N = \binom{n}{2}$, and $F(1) = D_{n-1}^-(f)$, which can be written as a determinant of size $(n-1) \times (n-1)$ with entries like $a_i - a_j$, $\pm a_i$ or $0$ (see Definition~\ref{def:dets}). Now expand this determinant completely, and also multiply out all $a_i - a_j$ expressions. In this way, we get an expansion of the form
\begin{equation} \label{F1-expansion}
	F(1) = \sum_{l=0}^{N} A_l = \sum_{i=1}^{T_n} \left( \pm \prod_{j=1}^{n-1} a_{s_{i,j}} \right)\!,
\end{equation}
where $s_{i,j}$ indicates the factor of the $i$th term of the expansion which comes from the $j$th column of the determinant, and $T_n$ denotes the number of terms in the expansion. Let $T_n^{(l)}$ be the number of terms corresponding to $A_l$, i.e.\ where the sum of the indices is $l$. Then we have $|A_l| \leq T_n^{(l)} H(f)^{n-1}$ from the expansion above, and altogether:
\begin{equation} \label{F1-abs-bound}
	|F(1)| \leq \sum_{l=0}^{N} |A_l| \leq \sum_{l=0}^{N} T_n^{(l)} H(f)^{n-1} = T_n H(f)^{n-1}.
\end{equation}

To give a bound on $T_n$, expand the determinant row by row starting from the last one, going back to the first. The last row gives a choice of $2$ entries. The next one contains $4$ entries, but one of them is excluded by the previous choice, so there are only $3$ possibilities. On each next row, there are $2$ more entries, but one more position is excluded, so there are at most $1$ more possibilities (the excluded entry can be a double-entry, which reduces the number even more, but it cannot be a zero entry). This gives a bound on the number of possibilities as $2 \cdot 3 \cdot 4 \cdot \ldots \cdot n$, i.e.\ $T_n \leq n!$.

Therefore we get from (\ref{F1-abs-bound}) that $\sum_{l=0}^{N} |A_l| \leq n! H(F)^{n-1}$. Continuing similarly as in the first part of the proof, but with $F$:
\begin{equation*}
	\frac{|F^{(k)}(1)|}{k!}
	= \left| \sum_{l=0}^{N} \binom{l}{k} A_l \right|
	\leq \binom{N}{k} \sum_{l=0}^{N} |A_l|
	\leq \binom{N}{k} n! H(f)^{n-1}.
\end{equation*}
However, we can improve the bound by a factor of two, noticing the symmetry of the determinant: reversing the coefficient sequence $a_0, a_1, \ldots, a_n$ does not change the value except possibly the sign, so any $\prod a_i$ term in the expansion has a pair $\prod a_{n-i}$ (maybe itself) with probably different sign. This means, according to the definition of the $D$-polynomials (Definition~\ref{def:d-poly}), that $A_l$ and $A_{n-l}$ have the same number of $\prod a_i$ terms, i.e.\ $T_n^{(l)} = T_n^{(N-l)}$, so we can pair them up to get a better bound:
\begin{gather*}
	\frac{|F^{(k)}(1)|}{k!}
	= \left| \sum_{l=0}^{N} \frac{1}{2} \left( \binom{l}{k} A_l + \binom{N-l}{k} A_{N-l} \right) \right| \leq \\
	\leq \frac{1}{2} \max_{l=0}^{N} \left( \binom{l}{k} + \binom{N-l}{k} \right) \sum_{l=0}^{N} T_n^{(l)} H(f)^{n-1}
	\leq \frac{1}{2} \binom{N}{k} n! H(f)^{n-1}.
\end{gather*}
We can use this to calculate:
\begin{equation*}
	|F(1 + \varepsilon) - F(1)|
	\leq \sum_{k=1}^{N} \frac{1}{2} \binom{N}{k} n! H(f)^{n-1} \varepsilon^k
	< \sum_{k=1}^{\infty} \left( \frac{N}{2} \varepsilon \right)^k n! H(f)^{n-1}.
\end{equation*}
By using (\ref{powers-aux}), we get:
\begin{equation*}
	\frac{1}{\varepsilon} = \frac{N}{2} \left( n! H(f)^{n-1} + 1 \right),
\end{equation*}
and we can finish the proof as usual by using (\ref{sqrt-aux}).
\end{proof}

\begin{theorem} \label{thm:bound-L}
For any root $\alpha$ of an expansive integer polynomial $f$:
\begin{equation*}
	\frac{1}{|\alpha| - 1} \leq \begin{cases}
		n L(f) + n & (\alpha \in \mathbb{R}), \\
		2 \binom{n}{2} L(f)^{n-1} + 2 \binom{n}{2} + 1 & (\alpha \in \mathbb{C} \setminus \mathbb{R}).
	\end{cases}
\end{equation*}
\end{theorem}

\begin{proof}
The proof is the same as that of Theorem~\ref{thm:bound-H}, except for the bound on $|f^{(k)}(\pm 1) / k!|$ and $|F^{(k)}(1) / k!|$. The first one (which is used for real $\alpha$) is replaced by the following:
\begin{equation*}
	\left| \frac{f^{(k)}(\pm 1)}{k!} \right|
	= \left| \sum_{j=0}^{n} (\pm 1)^{j-k} \binom{j}{k} a_j \right|
	\leq \binom{n}{k} L(f).
\end{equation*}

For the second one (if $\alpha$ is non-real), it is sufficient to replace $n! H(f)^{n-1}$ with $L(f)^{n-1}$, i.e.\ to show that $\sum_{l=0}^{N} |A_l| \leq L(f)^{n-1}$. Examine the expansion (\ref{F1-expansion}). Notice that each column of the determinant contains each coefficient $a_j$ at most once, so no $(s_{i,1},\, s_{i,2},\, \ldots,\, s_{i,n-1})$ index sequence can occur more than once in the expansion. Therefore, if we replace these $T_n$ sequences with all possible sequences from $(1,1,\ldots,1)$ to $(n-1,n-1,\ldots,n-1)$, we get an upper bound:
\begin{equation*}
	|F(1)| \leq \sum_{l=0}^{N} |A_l|
	\leq \sum_{i=1}^{T_n} \prod_{j=1}^{n-1} |a_{s_{i,j}}|
	\leq \sum_{i_1=1}^{n-1} \sum_{i_2=1}^{n-1} \cdots \sum_{i_{n-1}=1}^{n-1} \prod_{j=1}^{n-1} |a_{i_j}| = L(f)^{n-1}.
	\qedhere
\end{equation*}
\end{proof}

\subsection{Comparison} \label{sec:bounds:cmp}

We summarize the bounds on $\frac{1}{|\alpha| - 1}$ as follows (ignoring any negligible terms):

\vspace{1.5ex}
\begin{tabular}{|l|l|l|}
	\hline
	& $\alpha \in \mathbb{R}$
	& $\alpha \in \mathbb{C} \setminus \mathbb{R}$ \\ \hline
	Theorem~\ref{thm:bound-A} &
		$2^{n-1} |a_0|$ &
		$2^{\binom{n}{2}} |a_0|^{n-1}$ \\ \hline
	Theorem~\ref{thm:bound-AZ} &
		$2^{n-2} (|a_0| + |a_n|)$ &
		$2^{\binom{n-1}{2}} (|a_0| + |a_n|)^{n-1}$ \\ \hline
	Theorem~\ref{thm:bound-H} &
		$\binom{n+1}{2} H(f)$ &
		$\binom{n}{2} n! H(f)^{n-1}$ \\ \hline
	Theorem~\ref{thm:bound-L} &
		$n L(f)$ &
		$2 \binom{n}{2} L(f)^{n-1}$ \\ \hline
\end{tabular}
\vspace{1.5ex}

We can see that all bounds have the form $c_n A$ for real and $c_n A^{n-1}$ for non-real roots, and in each case, the latter is the larger (at least for $n \geq 3$).

We compare the four different bounds in each column. First, $|a_0| \leq |a_0| + |a_n| \leq L(f)$, $|a_0| + |a_n| \leq 2H(f)$ and $H(f) \leq L(f)$, so neither is better then any one below it. Next, since $|a_n| < |a_0|$ for expansive polynomials, the second row is strictly better than the first. For comparing $|a_0| + |a_n|$ and $H(f)$, we use the coefficient size relations (\ref{thm:exp-coeffs}) for expansive polynomials and Stirling's approximation ($\sqrt{2 \pi n} \left(\frac{n}{e}\right)^n \leq n! \leq e \sqrt{n} \left(\frac{n}{e}\right)^n$):
\begin{align*}
	H(f) &= \max_{k=0}^{n} |a_k|
	\leq \max_{k=0}^{n} \left( \binom{n-1}{k-1} |a_n| + \binom{n-1}{k} |a_0| \right) \leq \\
	&\leq \binom{n-1}{\frac{n-1}{2}} (|a_n| + |a_0|)
	\leq \frac{e}{\pi} \frac{2^{n-1}}{\sqrt{n-1}} (|a_n| + |a_0|).
\end{align*}
From this, one can see that in general, the third row is not better than the second one. The following shows that neither is the fourth row:
\begin{equation*}
	L(f) = \sum_{k=0}^{n} |a_k|
	\leq \sum_{k=0}^{n} \left( \binom{n-1}{k-1} |a_n| + \binom{n-1}{k} |a_0| \right)
	= 2^{n-1} (|a_n| + |a_0|).
\end{equation*}
And also, since $L(f) \leq (n+1) H(f)$, the fourth row is not better than the third either.

The conclusion is that in general, neither row is better than any other except that the second row is better than the first. It depends on the particular circumstances which of these bounds is the best. For example when the middle coefficients are much larger than the constant term and the leading coefficient (close to the extent that (\ref{thm:exp-coeffs}) permits), then $|a_0| + |a_n|$ may be the best measure, but otherwise $H(f)$ or $L(f)$.

\section{Further directions} \label{sec:final}

There are still several open questions regarding expansive polynomials and the results of this paper. In the future, we will try to answer the following:
\begin{itemize}
	\item The bounds of the form $c_n A$ and $c_n A^{n-1}$ given in Section~\ref{sec:bounds} are probably not the best, at least in the $c_n$ factors. We will try to find the best possible bounds, and prove sharpness by finding families of polynomials that have exactly the same expansivity gap asymptotically. Our conjecture is that the dependence on the coefficient size $A$ for fixed $n$ is asymptotically sharp. In the case of real roots, this can be easily proven e.g.\ by the polynomials $(A-1) x^n - A$.
	\item The $D$-conditions in Theorem~\ref{thm:exp-dets} inherently used the assumption that the coefficients are real, when splitting conditions like $|A| < B$ to $B \pm A > 0$. This is sufficient for our purposes with integer coefficients, but it would be an interesting question to generalize these conditions to arbitrary complex coefficients.
	\item There are $2n$ $D$-polynomials ($\widetilde{D}_k^\pm(f)(x)$ for $1 \leq k \leq n$ and for each sign), but we only fully understand $3$ of them: $\widetilde{D}_{n-1}^-(f)(x)$, and the related $\widetilde{D}_n^+(f)(x)$ and $\widetilde{D}_n^-(f)(x)$ (see Lemma~\ref{thm:d-polys}). We tried in vain to find the meaning of the others, e.g.\ how their roots relate to the roots of $f$.
	\item The bound in Theorem~\ref{thm:bound-H} involved the factor $n!$, which is a bound on $T_n$, i.e.\ the number of terms in the expansion of $D_{n-1}^-(f)$. This is not the best bound, and it would be an interesting question on its own to find better bounds on $T_n$, or even find an exact formula. The first few values for $n = 1, 2, \ldots$ are $T_n = 1, 2, 4, 12, 40, \ldots$
\end{itemize}

\end{document}